\documentclass{amsart}  % {{{1
\usepackage{amssymb, tikz, float, enumitem, mathtools, afterpage}
\usepackage[all, cmtip]{xy}
\usetikzlibrary{calc,decorations.markings,matrix,arrows,positioning}
\tikzset{>=latex}
\newcommand{\A}{\mathbb{A}}
\theoremstyle{plain} \newtheorem{thm}{Theorem}[section]
\newtheorem{prop}[thm]{Proposition}
\newtheorem{lem}[thm]{Lemma}

\theoremstyle{definition} \newtheorem{defn}[thm]{Definition}
\theoremstyle{remark} 
\theoremstyle{plain} \newtheorem{claim}{Claim}
\theoremstyle{plain} 

\newenvironment{claimproof} {
  \begin{proof}[Proof of claim]
  
  } {
  \end{proof}
  }

\DeclareMathOperator{\Con}{Con}
\DeclareMathOperator{\Clo}{Clo}

\numberwithin{equation}{section}  % number equations within sections
\renewcommand{\phi}{\varphi}
\renewcommand{\epsilon}{\varepsilon}

\usepackage{caption, subcaption, xparse}

\usepackage{framed}

\theoremstyle{definition} 
\theoremstyle{remark} 

\DeclareMathOperator{\Pol}{Pol}

\DeclareMathOperator{\face}{face}
\DeclareMathOperator{\cube}{cube}

\DeclareMathOperator{\Corners}{Corners}

\DeclareMathOperator{\faces}{Faces}
\DeclareMathOperator{\glue}{Glue}

\DeclareMathOperator{\tc}{TC}
\DeclareMathOperator{\cut}{Cut}

\DeclareMathOperator{\lines}{Lines}

\DeclareMathOperator{\refl}{Refl}

\DeclareMathOperator{\sym}{Sym}

\DeclareMathOperator{\tol}{Tol}

\DeclareMathOperator{\fin}{fin}

\DeclareMathOperator{\Cell}{Cell}

\def\nat{\mathbb{N}}
\def\A{\mathbb{A}}

\def\Meet{\bigwedge}
\def\Union{\bigcup}

\def\union{\cup}

\def\finsub{\subseteq_{\fin}}

\tikzset{myStyle/.style={baseline=(center.base), font=\small,
    every node/.style={inner sep=0.25em} }}

\NewDocumentCommand{\LinePic}{ O{} O{} O{1} }{ % {{{
  \begin{tikzpicture}[myStyle, scale=#3*1 ]
    \node (center) at (0,0.5) {\phantom{$\cdot$}}; % black magic
    \path (0,0)  node (s) {$#1$}
        ++(0,1)  node (n) {$#2$};
    \draw (n) -- (s);
  \end{tikzpicture}
}  % }}}

\newcommand{\SquareUnwrapped}[4]{ % {{{
  \node (center) at (0.5,-0.5) {\phantom{$\cdot$}}; % black magic
  \path (0,0)  node (nw) {$#2$}
      ++(1,0)  node (ne) {$#4$}
      ++(0,-1) node (se) {$#3$}
      ++(-1,0) node (sw) {$#1$};
  \draw (nw) -- (ne) -- (se) -- (sw) -- (nw);
}  % }}}
\NewDocumentCommand{\SquareXY}{ O{} O{} O{} O{} O{1} O{1} }{ % {{{
  \begin{tikzpicture}[myStyle, xscale=#5*1, yscale=#6*1 ]
    \SquareUnwrapped{#1}{#2}{#3}{#4}
  \end{tikzpicture}
}  % }}}
\NewDocumentCommand{\Square}{ O{} O{} O{} O{} O{1} }{ % {{{
  \SquareXY[#1][#2][#3][#4][#5][#5]
}  % }}}

\NewDocumentCommand{\SquareAxes}{ O{} O{} O{1} O{0} O{1} }{  % {{{
  \begin{tikzpicture}[myStyle, scale=#3*0.8] % scale to align with squares
    \node (center) at (0.5,0.5) {\phantom{$\cdot$}}; % black magic
    \draw (0,0) -- node[above]{$#1$} (1,0) node[right]{$#4$}
      (0,0) -- node[left]{$#2$} (0,1) node[above]{$#5$};
  \end{tikzpicture}
}  % }}}
\NewDocumentCommand{\CubeAxes}{ O{} O{} O{} O{} O{} O{} O{} }{  % {{{
  \begin{tikzpicture}[myStyle, scale=#4*0.85] % scale to align with cubes
    \node (center) at (0.5,0.75) {\phantom{$\cdot$}}; % black magic
    \draw (0,0) -- node[above]{$#1$} (1,0) node[right]{$#5$}
      (0,0) -- node[left]{$#2$} (0,1) node[above]{$#6$}
      (0,0) -- node[below left=-0.25em]{$#3$} (0.5,-0.5) node[below right=-0.2em]{$#7$};
  \end{tikzpicture}
}  % }}}

\newcommand{\CubeNodes}[8]{  % {{{
  \node at (0.75,-0.75) (center) {\phantom{$\cdot$}}; % black magic
  \path (0,0)  node (back_nw)      {$#2$}
      ++(1,0)  node (back_ne)      {$#4$}
      ++(0,-1) node (back_se)      {$#3$}
      ++(-1,0) node (back_sw)      {$#1$}
        (0.5,-0.5) node (front_nw) {$#6$}
      ++(1,0)  node (front_ne)     {$#8$}
      ++(0,-1) node (front_se)     {$#7$}
      ++(-1,0) node (front_sw)     {$#5$};
}  % }}}
\newcommand{\CubeUnwrapped}[8]{ % {{{
  \CubeNodes{#1}{#2}{#3}{#4}{#5}{#6}{#7}{#8}
  \draw (back_nw) -- (back_ne) -- (back_se) -- (back_sw) -- (back_nw)
    (front_nw) -- (front_ne) -- (front_se) -- (front_sw) -- (front_nw)
    (back_nw) -- (front_nw)
    (back_ne) -- (front_ne)
    (back_se) -- (front_se)
    (back_sw) -- (front_sw);
}  % }}}
\newcommand{\CubeDUnwrapped}[8]{ % {{{
  \CubeNodes{#1}{#2}{#3}{#4}{#5}{#6}{#7}{#8}
  \draw (front_nw) -- (front_ne) -- (front_se) -- (front_sw) -- (front_nw)
    (back_nw) -- (back_ne) (back_sw) -- (back_nw)
    (back_nw) -- (front_nw)
    (back_ne) -- (front_ne)
    (back_sw) -- (front_sw);
  \draw[densely dotted] (back_ne) -- (back_se) -- (back_sw)
    (back_se) -- (front_se);
}  % }}}
\NewDocumentCommand{\Cube}{ O{} O{} O{} O{} O{} O{} O{} O{} O{1} }{  % {{{
  \begin{tikzpicture}[myStyle, scale=#9*1 ]
    \CubeUnwrapped{#1}{#2}{#3}{#4}{#5}{#6}{#7}{#8}
  \end{tikzpicture}
}  % }}}

\NewDocumentCommand{\CubeDeep}{ O{} O{} O{} O{} O{} O{} O{} O{} O{1}  }{  % {{{
  \begin{tikzpicture}[myStyle, xscale=#9*1, yscale=1.5 ]
    \CubeUnwrapped{#1}{#2}{#3}{#4}{#5}{#6}{#7}{#8}
  \end{tikzpicture}
}  % }}}
\NewDocumentCommand{\CubeD}{ O{} O{} O{} O{} O{} O{} O{} O{} O{1} }{  % {{{
  \begin{tikzpicture}[myStyle, scale=#9*1 ]
    \CubeDUnwrapped{#1}{#2}{#3}{#4}{#5}{#6}{#7}{#8}
  \end{tikzpicture}
}  % }}}

\NewDocumentCommand{\DeltaZeroCubeD}{ O{} O{} O{} O{} O{} O{} O{} O{} O{1} }{  % {{{
  \begin{tikzpicture}[myStyle, scale=#9*1]
    \CubeDUnwrapped{#1}{#2}{#3}{#4}{#5}{#6}{#7}{#8}
    \draw (back_sw)  to[out=30,in=180-30] (back_se)
      (back_nw)  to[out=30,in=180-30] node[auto]{$\delta$} (back_ne)
      (front_sw) to[out=30,in=180-30] (front_se);
    \draw[dashed] (front_nw) to[out=30,in=180-30] (front_ne);
  \end{tikzpicture}
}  % }}}

\NewDocumentCommand{\DeltaOneCubeD}{ O{} O{} O{} O{} O{} O{} O{} O{} O{1} }{  % {{{
  \begin{tikzpicture}[myStyle, scale=#9*1]
    \CubeDUnwrapped{#1}{#2}{#3}{#4}{#5}{#6}{#7}{#8}
    \draw (back_sw)  to[out=120,in=240]node[left]{$\delta$} (back_nw)
      (back_se)  to[out=120,in=240]  (back_ne)
      (front_sw) to[out=120,in=240] (front_nw);
  \end{tikzpicture}
}  % }}}
\NewDocumentCommand{\DeltaTwoCubeD}{ O{} O{} O{} O{} O{} O{} O{} O{} O{1} }{  % {{{
  \begin{tikzpicture}[myStyle, scale=#9*1]
    \CubeDUnwrapped{#1}{#2}{#3}{#4}{#5}{#6}{#7}{#8}
    \draw (back_sw) to[out=180+30,in=180] node[auto,swap]{$\delta$} (front_sw)
      (back_se) to[out=0,in=30] (front_se)
      (back_nw) to[out=180+30,in=180] (front_nw);
    \draw[dashed] (back_ne) to[out=0,in=30] (front_ne);
  \end{tikzpicture}
}  % }}}

\begin{document}
% title and abstract {{{1
\title{Mal'cev Complexes}
\author{ Andrew Moorhead}

\address[]{ Institue f{\" u}r Algebra, TU Dresden
  }
\email[]{andrew\textunderscore paul.moorhead@tu-dresden.de}

\thanks{Andrew Moorhead has been funded by the European Research Council (Project POCOCOP, ERC Synergy
Grant 101071674). Views and opinions expressed are however
those of the author only and do not necessarily reflect those of the European Union or the European Research
Council Executive Agency. Neither the European Union nor the granting authority can be held responsible for them.
}

\begin{abstract}
It is well known that an equivalence relation is invariant under the basic operations of an algebra if and only if it is invariant under the unary polynomials of the algebra. We show that a higher arity version of this property holds for a higher dimensional analogue of an equivalence relation. It follows that the hypercommutator of arity $n$ for an algebra is determined by its $n$-ary polynomials. We construct examples to show that this fails for every arity of the term condition higher commutator. 

\end{abstract}

\maketitle % }}}1

\section{Introduction}

In this article we investigate some basic properties of certain relations that are connected to the commutator and higher commutator theory for general algebraic structures. These relations are equipped with a kind of higher dimensional rectangular geometry which allows for the definition of higher dimensional versions of the symmetric, reflexive, and transitive properties ordinarily defined for binary relations. In the successful pursuit of a categorical commutator theory, Janelidze and Pedicchio define and study a two dimensional version of these relations in \cite{janelidzepedicchio}. They were defined for higher dimensions to study the higher commutator in \cite{taylorsupnil} and are called there \emph{higher dimensional equivalence relations}. If invariant under the operations of an algebra, a higher dimensional equivalence relation is called a \emph{higher dimensional congruence}.

Mal'cev noticed early on that the congruences of an algebra are exactly the equivalence relations that are invariant under the action of its unary polynomials. This observation makes it easy to generate a congruence from a set of pairs: simply symmetrize the set, close under the unary polynomials, and take the transitive closure. A connected chain of so-called \emph{basic translations} witnessing that two elements are related by this transitive closure is now referred to as a \emph{Mal'cev chain}. Here we extend this idea to higher dimensions. Specifically, we will show that an $(n)$-dimensional equivalence relation of an algebra is an $(n)$-dimensional congruence exactly when it is invariant under the $(n)$-ary polynomials. From this we deduce that the $(n)$-dimensional congruences of an algebra $\A$ are determined by the $(n)$-ary polynomials of the algebra.

We then connect this result to the study of higher commutators and centralizer conditions. The (binary) commutator and the study of affine, nilpotent, and solvable algebras form a constellation of connected ideas that plays an important role in Universal Algebra. Smith first defined a general commutator for Mal'cev varieties in \cite{jdhsmith}. This initiated a line of research on the topic in the context of congruence modular varieties, much of which is collected in the volume of Freese and McKenzie \cite{fm}. Gumm independently develops the theory of the modular commutator alongside a geometrical framework in \cite{gumm}. Kearnes and Kiss wrote a detailed monograph developing different commutator theories outside of the context of congruence modular varieties \cite{kearneskiss}.

There are many candidate definitions for a commutator, all of which coincide for modular varieties \cite{fm}. An important feature of the Universal Algebra commutator is a powerful representation theory of abelian algebras. Herrmann showed that the abelian algebras in a modular variety are exactly the affine algebras, which are those algebras that are polynomially equivalent to a module \cite{heraffine}. Quackenbush later suggested a commutator which equates abelianness with such a representation \cite{quackenbushquasiaffine}, although in general the most one can hope for is quasiaffine, or embeddable into the reduct of an affine algebra. This suggestion of a `linear' commutator was developed by Kearnes and Szendrei in \cite{kearnesszendreirel}, where they show that it is equal to the symmetric term condition commutator in a Taylor variety. While there are many candidates for a commutator, the term condition has become the most widely used, and the results of Kearnes and Szendrei guarantee that this is usually a good choice for Taylor varieties. 

Bulatov generalized the term condition commutator to a commutator of higher arity in \cite{buldef}. Aichinger and Mudrinksi develop the basic theory of the higher commutator for Mal'cev varieties in \cite{aichmud}. In the same paper the authors also use the higher commutator to define the important class of `supernilpotent' algebras, which are algebras for which the higher commutator of some arity outputs the minimal congruence when evaluated at the full congruence. While supernilpotent algebras in general need not be nilpotent \cite{mooremoorhead}, they are necessarily so in Taylor varieties \cite{taylorsupnil}.

The proof that supernilpotent Taylor algebras are nilpotent proceeds by defining a new commutator (called the `hypercommutator') and then demonstrating that the term condition commutator coincides with the hypercommutator in a Taylor algebra when evaluated at a constant tuple of congruences. The hypercommutator is defined with a centrality condition that is quantified over the higher dimensional congruence relation that is obtained by taking a multidimensional transitive closure over the set of cubes that are used to define the usual term condition higher commutator. The properties of higher dimensional congruences ensure that the hypercommutator satisfies a natural inequality relating nested terms. An explicit use of properties of higher dimensional congruences (in this case two dimensional) is found in the Kearnes, Szendrei, and Willard characterization of the commutator for difference term varieties \cite{kearnesszendreiwillardcharacterizing}.

In this paper we establish another nice property satisfied by the $(n)$-ary hypercommutator: it is completely determined by the $(n)$-ary polynomials of the algebra under investigation. This has been known to hold for the binary modular commutator for some time (in \cite{mckenzieparks} McKenzie shows that it suffices to check centrality for binary polynomials). It is shown in \cite{higherkissterms} that the hypercommutator is always equal to the term condition higher commutator for a modular variety, so this older property of the binary commutator could be viewed as a special case of our higher commutator result (in fact McKenzie's result is stronger, but it implies that the binary commutator is determined by the binary polynomials in the modular setting). 

We begin with an informal demonstration of the ideas (Section \ref{sec:informal}) before defining the necessary notation to prove our results in general (Section \ref{sec:generalresultscomplex}). In Section \ref{sec:commutators}, we define the term condition commutator and hypercommutator, argue that the hypercommutator is determined by polynomials of the same arity, and provide a family of examples to demonstrate that this is not true of the term condition commutator in general.

\section{Informal exposition}\label{sec:informal}

Because this is a collection of results that holds for all finite dimensions, the notation we use is unfortunately somewhat cumbersome. Therefore, we will start with an informal and low dimensional exposition. Because we are permitting ourselves this momentary informality, we omit a subtlety relating to a relaxation of reflexivity and the constants that we use when considering polynomial clones. 

Let us first consider something elementary. If $\theta$ is a transitive relation on some set $A$ and $\gamma \in A^n$ is some tuple of elements of $A$ indexed by $n$ satisfying $\langle \gamma_i, \gamma_{i+1} \rangle \in \theta$ for all $i \in n-1$ (note that we use that a natural number is the set of its predecessors), then the pair $\langle \gamma_0, \gamma_{n-1} \rangle$ comprised of the two endpoints of the tuple is an element of $\theta$. This is of course a routine application of the transitive property. This situation may be depicted as in Figure \ref{fig:1dtransitive}. In anticipation of later terminology, we call such a chain of $\theta$-pairs a $(1)$-dimensional rectangular complex. We will say that such a complex is contractible when it is comprised of pairs that belong to a transitive relation.

\begin{figure}
\includegraphics[scale=1]{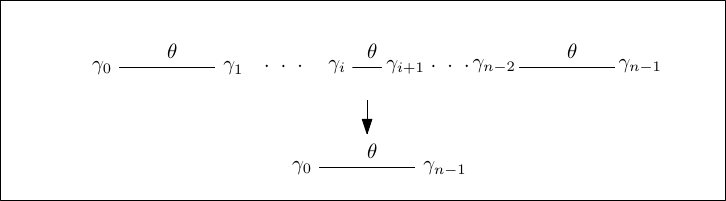}
\caption{Transitivity applied to a $(1)$-dimensional complex}\label{fig:1dtransitive}
\end{figure}

Now we consider the $(2)$-dimensional analogue. Informally, a $(2)$-dimensional relation on a set $A$ is a set of squares with vertices labeled by elements of $A$. We say that such a relation is $(2)$-transitive if we can glue squares together along a common edge either horizontally and vertically to obtain a new square of related elements. Put another way, all $(2)$-dimensional rectangular complexes consisting of such labeled squares are contractible, meaning the corners of the complex are also related. This is depicted in Figure \ref{fig:2dtransitive}. Here we are depicting that all labeled unit squares of some element $\gamma \in A^{n\times m}$ belong to a $(2)$-transitive relation $\theta$ and we conclude that the corners of the complex are also $\theta$-related. 

\begin{figure}
\includegraphics[scale=1]{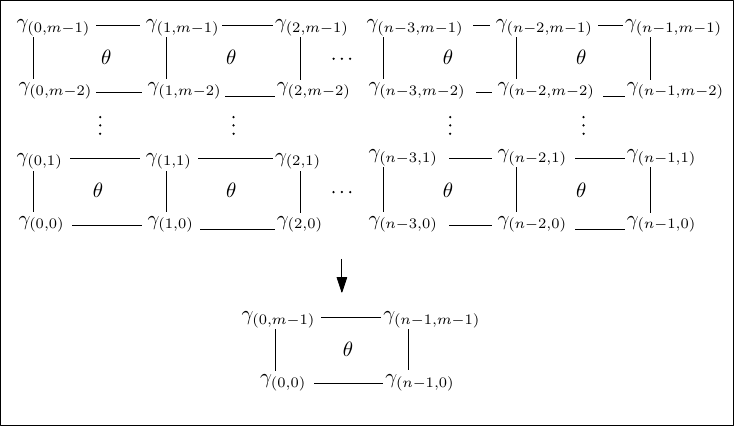}
\caption{Transitivity applied to a $(2)$-dimensional complex}\label{fig:2dtransitive}
\end{figure}

Returning to the $(1)$-dimensional case, let us now additionally suppose that $\A$ is an algebra and that $\theta$ is an equivalence relation on its underlying set $A$. In order to show that $\theta$ is a congruence of $\A$, it suffices to check its compatibility with the unary polynomials of $\A$. Indeed, suppose we are given (for example) a basic operation $t(x,y,z,w)$ of $\A$ with four arguments and four pairs $\langle a_0, b_0 \rangle, \dots , \langle a_3, b_3 \rangle$ that are $\theta$-related. Using the reflexivity of $\theta$ allows us to construct four $(1)$-dimensional rectangular complexes as depicted in Figure \ref{fig:1dunarypol}. Because we assume that $\theta$ is compatible with the unary polynomials of $\A$, it follows that the second complex is also a chain of $\theta$-related pairs. Now we can contract this chain to obtain the desired conclusion. 

\begin{figure}
\includegraphics[scale=1]{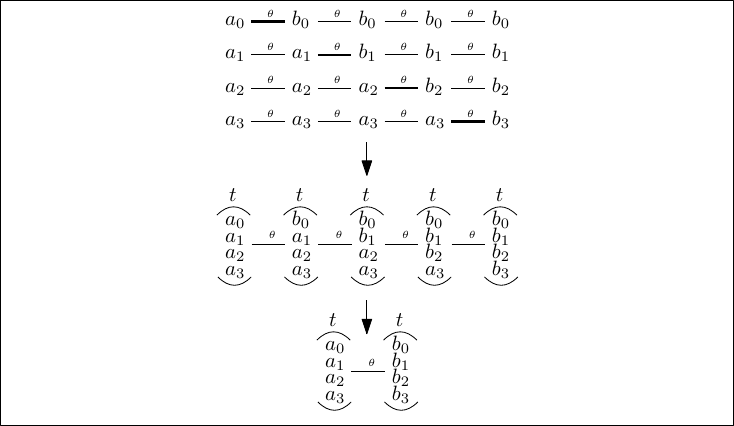}
\caption{Proving a $(1)$-dimensional equivalence relation is a congruence}\label{fig:1dunarypol}
\end{figure}

We again consider the $(2)$-dimensional analogue of this argument, the main components of which are depicted in Figure \ref{fig:2ybinarypol}. Suppose that $\theta$ is a $(2)$-dimensional equivalence relation, which we will define for all dimensions in Section \ref{sec:generalresultscomplex}, but for now describe informally. Here, the idea is that $\theta$ is a set of oriented squares with vertices labeled by elements of a set $A$, with the property that squares obtained from existing ones by reflecting over a horizontal or vertical line ($(2)$-symmetry), by duplicating a row or column ($(2)$-reflexivity), or contracting a $(2)$-dimensional rectangular complex ($(2)$-transitivity) again belong to $\theta$. 

Let us assume that we have such a $\theta$ and that $A$ is the underlying set of an algebra $\A$. We also assume that $\theta$ is compatible with the binary polynomials from $\A$. Suppose (for example) that $t(x,y,z,w)$ is an operation of $\A$ with four arguments. We want to show that $\theta$ is closed under $t$. At the top of Figure \ref{fig:2ybinarypol} we begin with four squares belonging to $\theta$. Each such square can be extended to a $(2)$-dimensional rectangular complex consisting of $\theta$ elements by repeatedly appealing to the $(2)$-dimensional reflexivity of $\theta$. These complexes are evaluated at the operation $t$. An inspection of the resulting complex reveals that each square belonging to this complex is the output of a binary polynomial of $\A$ evaluated at two elements belonging to $\theta$, so all such squares belong to $\theta$ by assumption. We then contract this complex to obtain the desired conclusion. 

\begin{figure}
\includegraphics[scale=1]{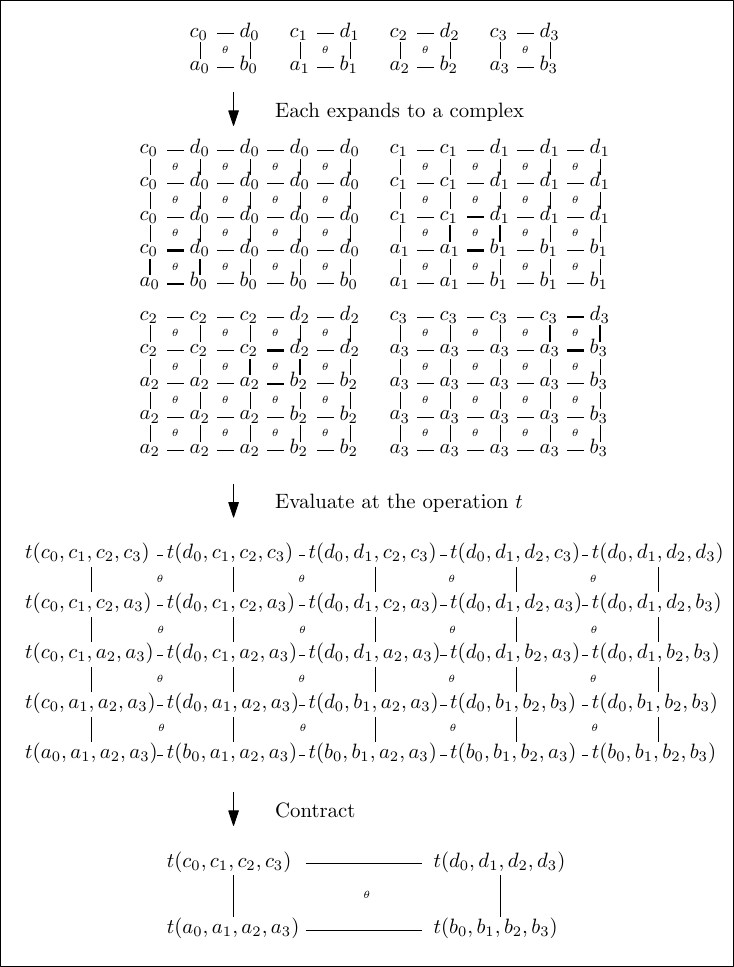}
\caption{Proving a $(2)$-dimensional equivalence relation is a $(2)$-dimensional congruence}\label{fig:2ybinarypol}
\end{figure}

We should mention that the other direction of these results is also easy. If an equivalence relation on the underlying set of an algebra is compatible with the basic operations of an algebra, then it is also compatible with the unary polynomials, because all but one of the arguments of a particular basic operation can be evaluated at constant pairs and such pairs belong to the equivalence relation because it is reflexive. This is equally easy to see in the $(2)$-dimensional case, because the $(2)$-dimensional reflexive property guarantees that constant squares belong to a $(2)$-dimensional equivalence relation (the general definitions  relax this to the constants belonging to the subalgebra determined by the labels of the vertices and we delay discussion of this detail until Section \ref{sec:generalresultscomplex}).

%The fact that the $(2)$-dimensional equivalence relations that are compatible with the binary polynomials of an algebra are exactly the ones that are compatible with all basic operations yields a method to generate them (see SECTION). In particular, the resulting relation is completely determined by the binary part of the polynomial clone of the algebra under consideration. This has the pleasant result that the binary commutator for an algebra is completely determined by its binary polynomials and in general that the $(n)$-ary higher commutator is completely determined by its $(n)$-ary polynomials. This is only true, however, if one defines the commutator using higher dimensional congruences. Section \ref{sec:commutators} contains some discussion about why this is a reasonable approach for commutators. 

\section{Higher dimensional formalism}\label{sec:generalresultscomplex}
The purpose of this section is to formalize and prove the results of the previous section for all dimensions bigger than or equal to one. Many of the definitions given here are also given in \cite{taylorsupnil} and we recommend the reader review the more detailed exposition given there. We consider the set of natural numbers as the set of all finite ordinals ordered by set membership. That is, the natural number $n = \{0, \dots, n-1\}$ is the set of its predecessors. Given a finite set of natural numbers $S$, we may equip the set of functions $2^S$ with a natural edge relation where two functions are related when there is exactly one argument $k \in S$ on which they differ. This is a standard way to define a $|S|$-dimensional hypercube. Now, given a nonempty set $A$, we may view the set $A^{2^S}$ as having a natural set of $|S|$-dimensional hypercube coordinates. For this reason, we say that any set of such \emph{labeled $|S|$-dimensional cubes} is a $|S|$-dimensional relation. 

We often wish to partially specify a set of $Q \subseteq S$ coordinates while working with hypercubes. There is a natural bijection between $A^{2^S}$ and $(A^{2^{S\setminus Q}})^{2^Q}$ which we define as 

\begin{align*}
\cut_Q: A^{2^S} &\to  (A^{2^{S\setminus Q}})^{2^Q}\\
\gamma &\mapsto 
\left\{
\left\langle 
f,
\{
\langle 
g, \gamma_{f \union g}
\rangle
: g \in 2^{S\setminus Q}
\}
\right\rangle
: f\in 2^Q
\right\}
\end{align*}

This function provides different ways of representing labeled cubes as lower dimensional cubes with vertices that are labeled by other labeled cubes. We denote the inverse of $\cut_Q$ by $\glue_Q$. There are some pictures of the $(4)$-dimensional situation provided in \cite{taylorsupnil}. We only rely in this paper on the mappings

\begin{align*}
\faces_i &\coloneqq \cut_{\{i\}}\\
\lines_i &\coloneqq \cut_{S\setminus \{i\}}
\end{align*}
for a particular coordinate $i \in S$. 

The output of $\faces_i$ represents the input as a $(1)$-dimensional cube with vertices labeled by $(|Q|-1)$-dimensional labeled cubes and so we will treat the output as a pair of lower dimensional faces. If we wish to argue about one of these lower dimensional faces in particular, we need to specify whether $i =0$ or $i=1$. This is accomplished by a superscript, for example if $\gamma \in A^{2^3}$, then $\faces_0^0(\gamma)$ and $\faces_0^1(\gamma)$ are two squares labeled by elements of $A$. The first is obtained restricting to the coordinates $\{(0,i,j): i, j \in 2 \}$ and the second is obtained by consider the coordinates $\{(1,i,j): i, j \in 2 \}$. 

Dually, the output of $\lines_i$ represents a labeled cube as a lower dimensional cube with vertices labeled by pairs instead of single elements. We will use this representation to articulate the centrality condition with which the higher commutator is defined (see Definition \ref{def:centrality}).

With this terminology, we can make the following sequence of definitions. 

\begin{defn}[cf.\ Definition 2.1 of \cite{taylorsupnil}]
Let $B$ be a nonempty set and let $R \subseteq B^2$ be a binary relation on $B$. We say that $R$ is a quasiequivalence relation on $B$ provided that each of the following conditions hold:
\begin{enumerate}
\item $\langle a, b \rangle \in R$ implies $\langle a,a \rangle, \langle b,b \rangle \in R$ (quasireflexivity),
\item $\langle a,b \rangle \in R$ if and only if $\langle b,a \rangle \in R$. (symmetry), and
\item $\langle a,b \rangle, \langle b,c \rangle \in R$ imply that $\langle a,c \rangle \in R$ (transitivity).
\end{enumerate} 
\end{defn}

\begin{defn}\label{def:highercon}[cf.\ Definition 2.2 of \cite{taylorsupnil}]
Let $\A $ be an algebra with underlying set $A$ and let $R \subseteq A^{2^S}$ be a $(|S|)$-dimensional relation for some $S\finsub \nat$.
\begin{enumerate}
\item  $R$ is said to be \textbf{$(S)$-reflexive}, \textbf{$(S)$-symmetric}, or \textbf{$(S)$-transitive} if \newline $\faces_i(R)$ is respectively quasireflexive, symmetric, or transitive on $A^{2^{S\setminus \{i\}}}$ for each $i \in S$.

\item $R$ is said to be a \textbf{$(|S|)$-dimensional equivalence relation} provided \newline $\faces_i(R)$ is a quasiequivalence relation on $A^{2^{S\setminus \{i\}}}$ for each $i \in S$. 
\item $R$ is said to be a \textbf{$(|S|)$-dimensional congruence} of $\A$ if it is a $(|S|)$-dimensional equivalence that is also compatible with the basic operation of $\A$.
\item $R$ is said to be a \textbf{$(|S|)$-dimensional tolerance} of $\A$ if it is $(S)$-reflexive, $(S)$-symmetric, and compatible with the basic operations of $\A$.

\end{enumerate}

\end{defn}

The higher dimensional versions of reflexivity and symmetry can be described in terms of certain unary operations. For each $i \in S$ and $j \in 2$, we define the maps $\refl_i^j : A^{2^S} \to A^{2^S}$ and $\sym_i^j: A^{2^S} \to A^{2^S}$
by 
\begin{align*}
\refl_i^j(h) &= \glue_{\{i\}}(\langle \faces_i^j, \faces_i^j \rangle ) \text{ and }\\
\sym_i(h) &= \glue_{\{i\}}(\langle \faces_i^1, \faces_i^0 \rangle).
\end{align*}
The following lemma is an easy consequence of the definitions. 

\begin{lem}\label{lem:reflsymclosure}
Let $A$ be a nonempty set and $S \finsub \nat$. Let $R \subseteq A^{2^S}$ be a $|S|$-dimensional relation. The following hold:
\begin{enumerate}
\item $R$ is $(S)$-reflexive if and only if $R$ is closed under $\refl_i^j$ for all $(i,j) \in S \times 2$, and
\item $R$ is $(S)$-symmetric if and only if $R$ is closed under $\sym_i$ for every $i \in S$.

\end{enumerate}

\end{lem}

Now we define rectangular complexes of such labeled cubes. Given a length $k \geq 1$ tuple of natural numbers $(n_0, \dots, n_{k-1})$ each greater than one, the product $n_0 \times \dots \times n_{k-1}$ is equipped with a natural rectangular graph structure, where the tuples $(i_0, \dots, i_{k-1})$ and $(j_0, \dots, j_{k-1})$ are edge related if they differ in exactly one coordinate $l$ and in this coordinate either $i_l = j_l +1 $ or $j_l = i_l +1$. We call such a graph a \emph{$k$-dimensional rectangular complex with dimensions $(n_0, \dots, n_{k-1})$}. Notice that if the dimensions $(n_0, \dots, n_{k-1})$ are all equal to $2$, then we recover a $k$-dimensional cube with vertex set $\underbrace{2 \times \dots \times 2}_{ k } = 2^k$. 

Given a set $A$ and a tuple of dimensions $(n_0, \dots, n_{k-1})$, we call an element 

\[
\gamma \in A^{n_0 \times \dots \times n_{k-1}}
\]
a \emph{labeled $k$-dimensional rectangular complex with dimensions $(n_0, \dots, n_{k-1})$}. The product $n_0 \times \dots \times n_{k-1}$ is a coordinate system for $\gamma$. For each coordinate tuple $f=(f_0, \dots, f_{k-1})$, we denote by $\gamma_f$ the value of $\gamma$ at $f$. Notice that if a labeled rectangular complex has dimensions $(n_0, \dots, n_{k-1})$ and each of these dimensions is equal to $2$, then we have recovered the definition of a labeled $(k)$-dimensional cube.

We use rectangular complexes to coordinatize and keep track of the higher dimensional transitive closures of relations that are coordinatized by higher dimensional cubes. Suppose that $\gamma \in A^{n_0 \times \dots \times n_{k-1}}$ is a labeled rectangular complex and suppose $f = (f_0, \dots, f_{k-1}) \in (n_0-1) \times \dots \times (n_{k-1}-1)$. We set 
$
\Cell_f(\gamma) \in A^{2^k}
$
to be the labeled $k$-dimensional cube satisfying 

\[
(\Cell_f(\gamma))_g = \gamma_{(f_0 + g_0, \dots, f_{k-1} + g_{k-1})}
\]
for each $g \in 2^k$. Every labeled $(k)$-dimensional rectangular complex may therefore be viewed as a collection of labeled $(k)$-dimensional cubes that coincide along certain $(k-1)$-dimensional faces. Indeed, it is immediate that

\[
\face_i^1(\Cell_{(f_0, \dots, f_i, \dots, f_{k-1})}(\gamma)) = \face_i^0(\Cell_{(f_0, \dots, f_i +1, \dots, f_{k-1})}(\gamma))
\]
for all choices of $i \in k$ and $f \in n_0 \times \dots \times n_{k-1}$ that make sense.

On the other hand, there is a labeled hypercube consisting of the corners of a labeled rectangular complex. Given $\gamma \in A^{n_0 \times \dots \times n_{k-1}}$, we set $\Corners(\gamma) \in A^{2^k}$ to be the labeled $k$-dimensional hypercube defined by 

\[
\Corners(\gamma)_g = \gamma_{(g_0(n_0-1), \dots, g_{k-1}(n_{k-1}-1))}
\]
for each $g \in 2^k$. 
\begin{prop}\label{prop:cornersinthetaifcellsintheta}
Let $A$ be a set, $k \geq 1$, and $\theta \subseteq A^{2^k}$ be a $k$-transitive set of labeled $k$-dimensional cubes. Suppose that $\gamma \in A^{n_0 \times \dots \times n_{k-1}}$ is a labeled $k$-dimensional rectangular complex. If $\Cell_f(\gamma) \in \theta$ for every $f \in (n_0 -1) \times \dots \times (n_{k-1} -1)$, then $\Corners(\gamma) \in \theta$. 
\end{prop}

\begin{proof}
The proof proceeds by induction on the poset of $k$-tuples of natural numbers, all greater than one, ordered by the rule $(n_0, \dots, n_{k-1}) \leq (m_0, \dots, m_{k-1})$ if and only if $n_i \leq m_i$ for each $i \in k$. The base case is clear, indeed, here every $n_i = 2$ and so $\Corners(\gamma)$ is equal to $\gamma$ which is equal to $\Cell_{(0,\dots,0)}(\gamma) \in \theta$. 

For the inductive step, we suppose that there exists $i \in k$ such that $n_i \neq 2$. Without loss, we assume that $i=0$. In this case we can decompose $\gamma$ into two complexes with strictly smaller dimensions and apply the inductive hypothesis to each. Define $\gamma^- \in A^{(n_0-1) \times n_1 \times \dots \times n_{k-1}}$ by

\[(\gamma^-)_f = \gamma_f 
\]
for $f \in (n_0-1) \times n_1 \times \dots \times n_{k-1}$ and define $\gamma^+$ by 

\[
(\gamma^+)_f = \gamma_{(n_0-2 + f_0, f_1, \dots, f_{k-1})}
\]
for $f \in 2\times n_1 \times \dots \times n_{k-1}$.
Both $\gamma^-$ and $\gamma^+$ satisfy the inductive hypothesis, so $\Corners(\gamma^-), \Corners(\gamma^+) \in \theta$. Furthermore, 

\[\faces_0^1(\Corners(\gamma^-)) = \faces_0^0(\Corners(\gamma^+)),\]
and because $\theta$ is a $(k)$-transitive relation, we deduce
\[\langle \faces_0^0(\Corners(\gamma^-)), \faces_0^1(\Corners(\gamma^+)) \rangle
\in \faces_0(\theta)
\]
or that, equivalently, $\Corners(\gamma) \in \theta$. 
\end{proof}

\begin{thm}\label{thm:compatiblewithnarypolys}
Let $\A$ be an algebra with underlying set $A$ and let $k \geq 1$. The $k$-dimensional congruences of $\A$ are exactly the $k$-dimensional equivalence relations $\theta$ on $A$ which are compatible with the $k$-ary polynomials of the subalgebra of $\A$ consisting of the elements of $\A$ that label cubes in $\theta$. 
\end{thm}

\begin{proof}
One direction is easy. Indeed, if $\theta \in A^{2^k}$ is a $(k)$-dimensional congruence of $\A$, then it is in particular compatible with the basic operations of $\A$ and $(k)$-reflexive. Suppose that $c \in A$ is the label of a vertex of some $\gamma \in \theta$ (i.e.\ that $c$ belongs to the subalgebra of elements that label cubes in $\theta$). A repeated application of the $(k)$-reflexivity of $\theta$ allows us to deduce that the cube with constant label $c$ is also an element of $\theta$. Therefore, $\theta$ is compatible with any $k$-ary polynomial of the subalgebra of $\A$ determined by these constants, because such polynomials come from terms of $\A$ with all but $k$-many arguments evaluated at such constants.

For the other direction, take $\theta \in A^{2^k}$ to be a $k$-dimensional equivalence relation that is compatible with the $k$-ary polynomials of $\A$ determined by elements of $A$ that label vertices of cubes in $\theta$. Take $t$ to be some basic operation of $\A$ of arity $n > k$. Take some $\gamma_0, \dots, \gamma_{n-1}$. We want to show that $t(\gamma_0, \dots, \gamma_{n-1}) \in \theta$. 

The intuition for the construction has been outlined in the earlier section. We first expand each $\gamma_i$ to a particular labeled rectangular complex of dimensions $(n+1)^k$, which we will call $\zeta_i(\gamma_i)$, as follows. For $i$ a natural number, let $\chi_{[i+1, \infty)}$ denote the characteristic function of the interval $[i+1, \infty)$. Now for any $\alpha \in A^{2^k}$ and $i \in n$, define the labeled $(k)$-dimensional rectangular complex $\zeta_i$ with dimensions $\underbrace{(n+1) \times  \dots \times (n+1)}_k$ as 

\[
\zeta_i(\alpha)_f \coloneqq \alpha_{(\chi_{[i+1, \infty)}(f_0),\dots, \chi_{[i+1, \infty)}(f_{k-1}) )},
\]
for each $f \in (n+1)^k$. 

Next we consider the complex $t(\zeta_0(\gamma_0), \dots, \zeta_{n-1}(\gamma_{n-1}))$. The full result will follow from the following sequence of claims. 
\begin{claim}
$\Corners(t(\zeta_0(\gamma_0), \dots, \zeta_{n-1}(\gamma_{n-1}))) = t(\gamma_0, \dots, \gamma_{n-1}).$
\end{claim}
\begin{claimproof} We have that
\[\Corners(t(\zeta_0(\gamma_0), \dots, \zeta_{n-1}(\gamma_{n-1})))=t(\Corners(\zeta_0(\gamma_0)), \dots, \Corners(\zeta_{n-1}(\gamma_{n-1}))),
\]
so we just need to see that $\Corners(\zeta_i(\gamma_i)))= \gamma_i$ for each $i \in n$. This follows from the definitions and the fact that $\zeta_i(\gamma_i))$ has dimensions $(n+1)^k$, while $i \in n$. 
\end{claimproof}
\begin{claim}

$\Cell_f(\zeta_{i}(\gamma_{i})) \in \theta$, for all $i \in n$ and $f \in n^k$.
\end{claim}

\begin{claimproof}\label{claim:cellbelongstotheta}
This follows from a repeated application of the $(k)$-reflexivity of $\theta$. Let us focus on a particular $i \in n$ and $f = (f_0, \dots, f_{k-1}) \in n^k$. We proceed by induction on the number of $j \in n$ with $f_j \neq i$. If this number is zero, then we have that $f = (\underbrace{i, \dots, i}_k)$ and in this case $\zeta_i(\gamma_i) = \gamma_i \in \theta$. Suppose that the result holds for $l$ many $j \in n$ and that $f_j \neq i $ for $l+1$ many $j$. Without loss, let us suppose that $f_0 \neq i$. By the inductive assumption we have that $\Cell_{f'}(\zeta_i(\gamma_{i})) \in \theta$, where $f' = (i, f_1, \dots, f_{k-1} )$. Because $f_0 \neq i$, it follows that $\chi_{[i+1, \infty)}(f_0) = \chi_{[i+1, \infty)}(f_0+1)$, so $\Cell_f(\zeta_{i}(\gamma_{i}))$ can be viewed as a pair of identical faces which are each equal to $\faces_0^0(\Cell_{f'}(\zeta_{i}(\gamma_{i})))$ or to $\faces_0^1(\Cell_{f'}(\zeta_{i}(\gamma_{i})))$. 

\end{claimproof}
\begin{claim}
$\Cell_f(t(\zeta_0(\gamma_0), \dots, \zeta_{n-1}(\gamma_{n-1}))) \in \theta$, for every $f \in n^k$.
\end{claim}

\begin{claimproof}
We have that 
\[
\Cell_f(t(\zeta_0(\gamma_0), \dots, \zeta_{n-1}(\gamma_{n-1}))) = t(\Cell_f(\zeta_0(\gamma_0)), \dots, \Cell_f(\zeta_{n-1}(\gamma_{n-1}))),
\]
so the claim will follow from the observation that at most $k$ many of the $\Cell_f(\zeta_{i}(\gamma_{i}))$ are nonconstant. To see this, let $f = (f_0, \dots, f_{k-1}) \in (n-1)^{k}$. Notice that $\Cell_f(\zeta_{i}(\gamma_{i}))$ depends on the coordinate $j \in k$ if and only if $f_j = i$ (else $f_j$ is too big or small for $\chi_{[i+1, \infty)}$ to detect the change from $f_j$ to $f_{j+1}$). Because $f$ has $k$ arguments, it follows that $k$ is largest number of nonconstant $\Cell_f(\zeta_{i}(\gamma_{i}))$ (this situation corresponds to injective $f$). Therefore,
\[
t(\Cell_f(\zeta_0(\gamma_0)), \dots, \Cell_f(\zeta_{n-1}(\gamma_{n-1})))
\]
is the output of one of the $k$-ary polynomials we assume $\theta$ to be compatible with. Because we showed in Claim \ref{claim:cellbelongstotheta} that each of the arguments belongs to $\theta$,  the claim is proved. 
\end{claimproof}
To finish the proof, we combine Claim 1, Claim 3, and Proposition \ref{prop:cornersinthetaifcellsintheta} to obtain that
\[
\Corners(t(\zeta_0(\gamma_0), \dots, \zeta_{n-1}(\gamma_{n-1}))) = t(\gamma_0, \dots, \gamma_{n-1}) \in \theta.
\]

\end{proof}

We conclude this section by generalizing the notion of a Mal'cev chain to all dimensions $k \geq 1$. Because of how we have chosen to define higher dimensional reflexivity, this procedure produces a higher dimensional congruence with constants determined by the subalgebra generated by the elements labeling the vertices of the generators. This discussion further develops the notation and discussion on pages 9-10 of \cite{taylorsupnil}.

Take $S \subseteq \nat$ be a finite set with at least one element and let $X\subseteq A^{2^S}$. We respectively define the $(|S|)$-dimensional congruence and $(|S|)$-dimensional tolerance of $\A$ generated by $X$ as

\begin{align*}
\Theta_S(X) &= \Meet \{R: R \text{ is a $(|S|)$-dimensional congruence and } X\subseteq R \}\\
\tol_S(X) &= \Meet \{R: R \text{ is a $(|S|)$-dimensional tolerance and } X\subseteq R \}.
\end{align*}
Suppose that $S = \{i_0, \dots, i_{n-1}\} $ is an enumeration of the elements of $S$. Let $Y \subseteq A^{2^S}$ be a $(|S|)$-dimensional relation. For $i \in S$ set 

\[ Y^{\circ_i } =  \glue_i( \faces_i(Y)^{\circ}),\]
where $\faces_i(Y)^{\circ}$ is the transitive closure of $\faces_i(Y)$ when interpreted as a binary relation. We recursively define

\begin{enumerate}
\item $\tc_{0}(Y) = Y^{\circ_{i_0}}$, and 
\item $\tc_{j}(Y) = \left(\tc_{j-1}(Y)\right)^{\circ_{i_{j\mod n}}}$, for $j > 0$.

\end{enumerate}
Finally, set $\tc(Y) = \Union_{j \in \nat} \tc_{i_j}(Y)$. 

The following lemma also appears in \cite{taylorsupnil}. Because it is not used in this paper, we omit item \emph{(3)} from the statement.

\begin{lem}[cf.\ Lemma 2.9 of \cite{taylorsupnil}] \label{prop:hcongenerate} 
Let $\A$ be an algebra and $S \finsub \nat$. The following hold.
\begin{enumerate}
\item If $R$ is a $(|S|)$-dimensional tolerance of $\A$, then $R^{\circ_i}$ is a $(|S|)$-dimensional tolerance of $\A$ and $R \subseteq R^{\circ_i}$.
\item
$\Theta_S(X) = \tc(\tol_S(X))$, for all $X \subseteq A^{2^S}$.

\end{enumerate}

\end{lem}

Suppose now that we are given a set $G \subseteq A^{2^S}$ and wish to explicitly characterize $\theta_S(G)$ in a manner similar to the characterization of ordinary congruences given by Mal'cev chains. Let $C(G)$ be the subalgebra generated by the elements of $A$ labeling cubes in $G$. In view of Lemma \ref{prop:hcongenerate}, we should first generate a $(|S|)$-tolerance in the algebra with operations from $\Pol_{|S|}(C(G))$ and then take an iterated transitive closure (item \emph{(2)} ensures that the result will be a $(|S|)$-dimensional congruence of the polynomial algebra). In view of Theorem \ref{thm:compatiblewithnarypolys}, this relation is also compatible with the basic operations of $\A$. We collect these observations in the following theorem. 

\begin{thm}\label{thm:generatehdimcon}
Let $\A$ be an algebra with underlying set $A$. Let $S \subseteq \nat$ be a finite set with at least one element. Let $G \subseteq A^{2^S}$ be a set of $(|S|)$-dimensional cubes labeled by elements of $A$. Let $C(G) \leq \A$ be the subalgebra of $\A$ generated by the elements of $A$ that label vertices of elements in $G$. The following is a procedure to produce $\Theta_S(G)$.

\begin{itemize}
\item Symmetrically and reflexively close $G$ (close $G$ under $\sym_i$ and $\refl_i^j$ for all $i \in S$ and $j \in 2$). 
\item Close the resulting set with $\Pol_{|S|}(C(G))$
\item Close the resulting set with $TC$. 
\end{itemize}
\end{thm}

\section{Different commutators}\label{sec:commutators}

In this section we define the term condition commutator and what we call the hypercommutator and discuss how our work in the previous sections impacts these commutators. Informally, the term condition commutator corresponds to a condition that is quantified over higher dimensional tolerances (which are usually called matrices in the literature), while the hypercommutator is a condition that is quantified over the higher dimensional congruence that generated by the matrices used to define the term condition commutator.

Let $\A$ be an algebra and $S \finsub \nat$ with $|S|\geq 1$. For each $i \in S$ define 
$
\cube_i: A^2 \to A^{2^S}
$
by 
\[
\cube_i(x,y)_f = 
\begin{cases}
x & \text{if } f_i = 0 \text{, and}\\
y & \text{if } f_i =1.
\end{cases}
\]
From the context it should be clear what the dimension of $\cube_i(x,y)$ is.

\begin{defn}\label{def:tolmatdeltadefinition}
Let $\A$ be an algebra and $S \finsub \nat$ with $|S|\geq 1$. Let $\{\theta_i\}_{i \in S} \subseteq \Con(\A)$ be an $S$-indexed set of congruences. Set

\begin{align*}
M(\{\theta_i\}_{i \in S}) &= \tol_S\bigg( \Union_{i\in S} \cube_i(\theta_i)
 \bigg) \text{, and}\\
\Delta(\{\theta_i\}_{i \in S}) &= \Theta_S\bigg( \Union_{i\in S} \cube_i(\theta_i) \bigg).
\end{align*}

\end{defn}

\begin{defn}\label{def:centrality}
Let $\A$ be an algebra, $\delta \in \Con(\A)$, $S\finsub \nat$ with $|S| \geq 2$, and $i\in S$. We say that a $(|S|)$-dimensional relation $R$ on $A$ has  \textbf{$(\delta,i)$-centrality} if there is no $\gamma\in R$ such that exactly $2^{|S|-1}-1$ many vertices of $\lines_i(\gamma)$ are labeled by $\delta$-pairs.
\end{defn}

\begin{defn}\label{def:tccomm}
Let $\A$ be an algebra and $S \finsub \nat$ with $|S|\geq 2$. Let $\{\theta_i\}_{i \in S} \subseteq \Con(\A)$ be an $S$-indexed set of congruences. Let $k$ be the greatest element of $S$. We define

\begin{align*}
[\{\theta_i\}_{i \in S}]_{TC} &= 
\Meet \{\delta: M(\{\theta_i\}_{i \in S}) \text{ has } (\delta, k)\text{-centrality} \} \\
[\{\theta_i\}_{i \in S}]_{H} &= \Meet \{\delta: \Delta(\{\theta_i\}_{i \in S}) \text{ has } (\delta, k)\text{-centrality} \}.
\end{align*}
We call these operations the $(|S|)$-ary \textbf{term condition commutator} and \textbf{hypercommutator}, respectively. In case $S =n $, we use the notation
$[\theta_0, \dots, \theta_{n-1}]_{TC}$ and $[\theta_0, \dots, \theta_{n-1}]_H$ for these operations.
\end{defn}

With these definitions, we can now state a corollary of Theorem \ref{thm:generatehdimcon}. 

\begin{thm}\label{thm:hyperdeterminedbypoly}
Let $\A$ be an algebra and let $k \geq 1$ be a natural number. The hypercommutator of arity $k$ for $\A$ is determined by the $k$-ary polynomials of $\A$. 
\end{thm}

On the other hand, Theorem \ref{thm:hyperdeterminedbypoly} does not hold for the term condition commutator. Let $\mathbb{C} = \langle \nat; t(x,y,z) \rangle$ be the algebra with a countably infinite underlying set and a single basic operation $t(x,y,z)$ defined as 

\begin{align*}
t(x,y,z) =
\begin{cases}
3 & (x,y,z) = (0,0,0)\\
3 & (x,y,z) = (1,2,0) \\
s(x,y,z)  & \text{otherwise}
\end{cases},
\end{align*} 
where $s(x,y,z)$ is an injection from $\nat^3$ into $\nat \setminus \{0,1,2,3\}$. The algebra $\mathbb{C}$ is not abelian, as witnessed by 

\[
\SquareXY[t(0,0,0)][t(1,2,0)][t(0,0,1)][t(1,2,1)][2][1] \in M(1,1).
\]
However, the algebra with underlying set $\nat$ and operations $\Pol_2(\mathbb{C})$ is term condition abelian. To see this, we first show that the following is true.

\begin{lem}\label{lem:example1}
Consider the algebra $\mathbb{C}$ defined above. Suppose that there exist $a,b, d \in \nat$ such that  $b \neq d$ and
\[
\Square[a][a][b][d] \in M(1,1)
\]
is produced from the generators of $M(1,1)$ by the term operation $s \in \Clo(\mathbb{C})$. It follows that $s$ has a term tree with $t(u,v,w)$ as a subterm, with some independent $u,v,w$ occurring among the variables of $s$. 
\end{lem}

\begin{proof}
The proof proceeds by induction on the height of the term tree for $s$. The basis of the induction is trivial, because no generator of $M(1,1)$ satisfies the assumptions. Suppose that $s$ has a term tree with nonzero height. The root of the term tree is $t$ with children $s_1, s_2, s_3$. Let us suppose that each of these terms corresponds to the matrices 

\[
\Square[a_1][c_1][b_1][d_1], \Square[a_2][c_2][b_2][d_2], \Square[a_3][c_3][b_3][d_3]
\]
when evaluated. We assume that $t(a_1,a_2,a_3) = t(c_1, c_2, c_3) =a$ and that $t(b_1, b_2, b_3) \neq t(d_1, d_2, d_3)$. There are two cases to consider.

In the first case, suppose that $(a_1, a_2, a_3) = (c_1, c_2, c_3)$. In order for $t(b_1, b_2, b_3) \neq t(d_1, d_2, d_3)$, it is necessary that $(b_1, b_2, b_3) \neq (d_1, d_2, d_3)$, so the inductive assumption applies to one of the squares corresponding to $s_1, s_2, s_3)$ and we obtain the desired conclusion. 

In the second case, suppose that $(a_1, a_2, a_3) \neq (c_1, c_2, c_3)$. We assume that $t(a_1, a_2, a_3) = t(c_1, c_2, c_3)$, and because the domain of $t$ contains exactly one pair of inputs that witnesses a failure of injectivity, it must be that $(a_1, a_2, a_3) = (0,0,0)$ and $(c_1, c_2, c_3) = (1,2,0)$ or vice versa. Because $t$ does not output the values $0,1,2$, it is necessary that $s$ be of height one, and so the conclusion also holds in this case. 
\end{proof}

It follows from the above lemma that the algebra with underlying set $\mathbb{N}$ and operations $\Pol_2(\mathbb{C})$ is abelian. Indeed, no operation formed from the composition of binary polynomials of $\mathbb{C}$ can satisfy the conclusion of the lemma and so no violation of abelianness can be produced with binary polynomials. However, it is easy to see that a failure of abelianness is detectable with the binary hypercommutator, because the two matrices

\[
\SquareXY[t(0,0,0)][t(1,0,0)][t(0,0,1)][t(1,0,1)][2][1], 
\SquareXY[t(1,0,0)][t(1,2,0)][t(1,0,1)][t(1,2,1)][2][1]
\]
can be produced with binary polynomials and they share a common edge. We can generalize this kind of argument to produce an algebra that has $k$-arity commutator behavior that is not detectable with just its $(k)$-arity polynomials. The idea is to define an algebra with a single basic operation that has enough injectivity failures to ensure a failure of the condition that $[1, \dots, 1]_{TC} =0$, while also ensuring that this behavior is absent from lower arity polynomial operations. 

\begin{defn}
Let $k \geq 2$ and let $t_k: \nat^{k+1} \to \nat$ be defined by 
\[
t(x_0, x_1, x_2, \dots, x_k) = k+2
\]
for all $(x_0, x_1, x_2, \dots, x_k) \in \{(0,0), (1,2)\} \times \{0,3\} \times \dots \times \{0, k+1 \} \setminus \{(1,2, \dots, k+1)\}
$
and otherwise equal to some injective into $\nat \setminus \{0,\dots, k+2 \}$. Let $\mathbb{C}_k$ be the algebra with underlying set $\nat $ and the single basic operation $t_k$. 
\end{defn}

\begin{thm}
Let $k \geq 2$. The algebra $\mathbb{C}_k = \langle \nat ; t_k \rangle $ fails the condition $[\underbrace{1, \dots, 1}_{k} ]_{TC} = 0$, but the algebra with operations consisting of the $k$-ary polynomials of $\mathbb{C}_k$ satisfies $[\underbrace{1, \dots, 1}_k]_{TC} = 0$. 
\end{thm}

\begin{proof}
First, consider 
\[
\eta = 
t_k
\big(
\cube_0(0,1),
\cube_0(0,2),
\cube_1(0,3), \dots ,
\cube_{k-1}(0,k+1)
\big) \in M(\underbrace{1,\dots, 1}_k).
\]
The definition of $t_k$ gives that $\eta_f = k+2$ for all $f \in 2^k$ except for the function with constant value $1$, so $\eta$ witnesses that $[\underbrace{1, \dots, 1}_k]_{TC} \neq 0$ for the algebra $\mathbb{C}_k$. 

We will now characterize the terms of $\mathbb{C}_k$ that are capable of generating a failure of $[\underbrace{1, \dots, 1}_k]_{TC}=0$. We claim that any term generating such a failure must contain an instance of $t_k$ as a subterm which does not identify any variables. The proof follows the idea from the low arity example above. No generators of $M(\underbrace{1, \dots, 1}_k)$ witness a failure, so our claim holds for terms of height equal to zero. Now suppose that $s$ is a term with nonzero height and variables $x_0, \dots x_{m-1}$ for $m \geq k$ (there must be at least $k$ variables to generate a failure of $[\underbrace{1, \dots, 1}_k]_{TC}=0$). The term tree for $s$ has $t_k$ as a root and subterms $s_0, \dots, s_k$. Suppose that $\gamma = s(a_0, \dots, a_{m-1})$ witnesses a failure of $[\underbrace{1, \dots, 1}_k]_{TC} = 0$, for some generators $a_0, \dots, a_{m-1}$. By definition, this means that all but one of the pairs of $\lines_{k-1}(\gamma)$ are constant pairs, i.e.\ that $\lines_{k-1}(\gamma)_g $ is constant for every $g \in 2^{k-1}$ except for a particular $g^* \in 2^{k-1}$. Consider $\gamma_0, \dots, \gamma_{k}$ to be the result of evaluating the generators of $\gamma$ at the variables occurring in $s_0, \dots, s_{k}$, respectively. That is, we have that 
$
\gamma = t_k(\gamma_0, \dots, \gamma_{k}).
$
There are two cases to consider.

In the first case, each of the pairs $\lines_{k-1}(\gamma)_g$ for $g \in 2^{k-1} \setminus \{g^*\}$ is constant because $\lines_{k-1}(\gamma_i)_g$ is also a constant pair, for each $i \in k+1$ and $g \in 2^{k-1} \setminus \{g^*\}$. We assume that $\lines_{k-1}(\gamma)_{g^*}$ is not a constant pair, so in this case at least one of the $\gamma_i$ must also have the property that $\lines_{k-1}(\gamma_i)_{g^*}$ is not constant, which in turn means that $\gamma_i$ witnesses a failure of $[\underbrace{1, \dots , 1}_k]_{TC} = 0$. It follows inductively that $s$ has subterm $t_k$ depending on $k+1$ many variables. 

In the second case, there exists $i \in k+1$ and $g' \in 2^{k-1} \setminus \{g^*\}$ so that $\lines_{k-1}(\gamma_i)_{g'}$ is not a constant pair. We are assuming that 
 $\lines_{k-1}(\gamma)_{g'}$ is a constant pair. This means that the terms $s_0, \dots, s_k$ must collectively output distinct values for which $t_k$ can fail to be injective. This is impossible unless each of the subterms $s_0, \dots, s_k$ is a variable.  We must now argue that the number of variables is equal to $k+1$. 

The number of variables cannot be fewer than $k$. This is because each of the generators of $M(\underbrace{1, \dots, 1}_k)$ is a cube that depends on at most one coordinate, so a term operation with fewer than $k$ independent variables will output a cube that does not depend on at least one coordinate. Such a cube cannot exhibit a failure of $[\underbrace{1, \dots, 1}_k]_{TC} = 0$. So, we have only to consider the case where $s$ is the term operation $t_k$ with exactly two variables identified. Now, the only repeated value in the operation table for $t_k$ is the value $k+2$ and it is only available as an output on $\{(0,0), (1,2)\} \times \{0,3\} \times \dots \times \{0, k+1\}$. It is clear that $k+1$ generator cubes are needed to make this work. 

So, we have characterized which terms of $\mathbb{C}_k$ are capable of producing failures of $[\underbrace{1, \dots , 1}_k]_{TC}$. None of these terms can be produced with functions from $\Pol_k(\mathbb{C}_k)$. We have therefore shown that the $k$-ary term condition commutator is not in general determined by $k$-ary polynomials. 

\end{proof}
\bibliographystyle{plain}
\bibliography{refs}
\begin{center}
  \rule{0.61803\textwidth}{0.1ex}   % 1/(golden ratio)
\end{center}
\subjclass{MSC 08A40 (08A05, 08B05)}
\end{document}